\documentclass[11pt,oneside,a4paper]{article}

\usepackage[latin1]{inputenc}
\usepackage[english]{babel}
\usepackage{amsmath,amssymb,amsthm,epsfig,color,graphicx}
\usepackage{hyperref}
\usepackage{vmargin}

\newtheorem{theorem}{Theorem}
\newtheorem{proposition}[theorem]{Proposition}
\newtheorem{lemma}[theorem]{Lemma}

\let\varep=\varepsilon

\title{Existence results for the
Klein-Gordon-Maxwell equations in higher dimensions with critical
exponents}
\author{Paulo C. Carri\~{a}o\\
\footnotesize{Departamento Matem\'atica, UFMG, 30161-970 Belo
Horizonte -- MG, Brasil}\\
\footnotesize{\texttt{carrion@mat.ufmg.b}r}\\
\and
Patr\'{i}cia L. Cunha \thanks{Supported by CAPES/Brazil.}\\
\footnotesize{Departamento Matem\'atica, UFSCar, 13565-905 S\~{a}o Carlos -- SP, Brasil}\\
\footnotesize{\texttt{patcunha80@gmail.com}}\\
 \and
Olimpio H. Miyagaki\thanks{Supported in part by CPNq/Brazil and INCTMat-CNPq/Brazil.}\\
\footnotesize{Departamento Matem\'atica, Universidade Federal de
Vi\c{c}osa, 36570-000 Vi\c{c}osa -- MG, Brasil}\\
\footnotesize{\texttt{olimpio@ufv.br}} }

\begin{document}
\maketitle

\begin{abstract}
In this paper we study the existence of radially symmetric
solitary waves in $\mathbb{R}^{N}$ for the nonlinear Klein-Gordon
equations coupled with the Maxwell's equations when the
nonlinearity exhibits critical growth. The main feature of this
kind of problem is the lack of compactness arising in connection
with the use of variational methods.
\end{abstract}

\text{\footnotesize{\textit{Keywords}: Klein-Gordon-Maxwell
system; radially symmetric solution; critical growth}}


\section{Introduction}
This article concerns the existence of solutions for the
Klein-Gordon-Maxwell $(\mathcal{KGM})$ system in $\mathbb{R}^{N}$
with critical Sobolev exponents
\begin{eqnarray}
-\Delta u+[m_{0}^{2}-(\omega +\phi)^{2}]u=\mu|u|^{q-2}u+|u|^{2^*-2}u \quad \text{in } \mathbb{R}^{N}, \label{system1} \\
 \Delta \phi=(\omega+\phi)u^{2} \quad \text{in } \mathbb{R}^{N}, \label{system2}
 \end{eqnarray}
\noindent where $2<q<2^{*}=2N/(N-2)$, $\mu>0$ ,$m_{0}>0$ and
$\omega\neq 0$ are real constants and also
$u,\phi:\mathbb{R}^{N}\rightarrow \mathbb{R}$.

Such system has been first introduced by Benci and Fortunato
\cite{Benci-Fortunato-2001} as a model which describes nonlinear
Klein-Gordon fields in three-dimensional space inte\-racting with
the eletromagnetic field. Further, in the quoted paper
\cite{Benci-Fortunato} they proved existence of solitary waves of
the couplement Klein-Gordon-Maxwell equations when the
nonlinearity has subcritical behavior.

Some recent works have treated this problem still in the
subcritical case and we cite a couple of them.

D'Aprile and Mugnai \cite{D'Aprile-Mugnai} established the
existence of infinitely many radially symmetric solutions for the
subcritical $(\mathcal{KGM})$ system in $\mathbb{R}^{3}$. They
extended the interval of definition of the power in the
nonlinearity exhibited in \cite{Benci-Fortunato}.  For related
works, see \cite{Georgiev-Visciglia} and \cite{Mugnai}.

Non-existence results and a treatment of the $(\mathcal{KGM})$
system in bounded domains can be found in (\cite{Cassani},
\cite{D'Aprile-Mugnai-nonexistece},
\cite{d'Avenia-Pisani-Siciliano},
\cite{d'Avenia-Pisani-Siciliano2} and references therein).

With this \textit{Ansatz} Cassani \cite{Cassani} proved the
existence of nontrivial radially symmetric solutions in
$\mathbb{R}^{3}$ for the critical case. He was able to show that
\begin{itemize}
 \item if $|m_{0}|>|\omega|$ and $4<q<2^{*}$, then for each $\mu>0$ there exists at least a radially symmetric
 solution for system (\ref{system1})-(\ref{system2}).
 \item if $|m_{0}|>|\omega|$ and $q=4$, then system (\ref{system1})-(\ref{system2}) also has at least a radially
 symmetric solution by supposing $\mu $ sufficiently large.
\end{itemize}

The goal of this paper is to complement Theorem 1.2 from Cassani
in \cite{Cassani} and also extend it in higher dimensions as
follows

\begin{theorem}\label{principal}
Assume either $|m_{0}|>|\omega|$ and $4\leq q<2^{*}$ or
$|m_{0}|\sqrt{q-2}>|\omega|\sqrt{2}$ and $2<q<4$.

Then system (\ref{system1})-(\ref{system2}) has at least one
radially symmetric (nontrivial) solution $(u,\phi)$ with $ u \in
H^{1}(\mathbb{R}^{N})$ and $\phi\in
\mathcal{D}^{1,2}(\mathbb{R}^{N})$ provided that
\begin{itemize}
    \item [i)]$N=4$ and $N\geq 6$ for $2<q<2^*$ if $\mu>0$;
    \item [ii)] $N=5$ and either $2<q <\frac{8}{3}$ if $\mu>0$ or $\frac{8}{3}\leq q
    <2^*$ if $\mu$ is sufficiently large;
    \item [iii)] $N=3$ and either $4<q<2^*$ if $\mu>0$ or $2<q\leq
    4$ if $\mu$ is sufficiently large.
\end{itemize}
\end{theorem}

In order to get this result we will explore the Br\'{e}zis and
Nirenberg technique and some of its variants. See e.g.
\cite{Olimpio}.


\section{Preliminary Results}

We want to find solutions of the system
(\ref{system1})-(\ref{system2}) where $u\in H^{1}(\mathbb{R}^{N})$
and $\phi\in \mathcal{D}^{1,2}(\mathbb{R}^{N})$.

Here $H^{1}\equiv H^{1}(\mathbb{R}^{N})$ denotes the usual Sobolev
space endowed with the norm
\begin{eqnarray} \|u\|^{2}=\int_{\mathbb{R}^{N}}(|\nabla u|^2+u^2)dx
\end{eqnarray}
\noindent and $\mathcal{D}^{1,2}\equiv
\mathcal{D}^{1,2}(\mathbb{R}^{N})$ denotes the completion of
$\mathcal{C}_{0}^{\infty}(\mathbb{R}^{N})$ with respect to the
norm
\begin{eqnarray} \|u\|_{\mathcal{D}^{1,2}}^{2}=\int_{\mathbb{R}^{N}}|\nabla u|^2dx.
\end{eqnarray}

The $(\mathcal{KGM})$ system are the Euler-Lagrange equations
related to the functional
$$F:H^{1}\times \mathcal{D}^{1,2}
\rightarrow \mathbb{R}$$ defined as
\begin{eqnarray}\label{F} && F(u,\phi) =\frac{1}{2}\int_{\mathbb{R}^{N}}(|\nabla u|^2-|\nabla \phi|^2+[m_{0}^2-(\omega+\phi)^2 ]u^2)dx+\nonumber\\
&&\hspace{1.5cm}-\frac{\mu}{q}\int_{\mathbb{R}^{N}}|u|^q
dx-\frac{1}{2^{*}}\int_{\mathbb{R}^{N}}|u|^{2^{*}} dx,
\end{eqnarray}
\noindent which by standard arguments is $C^{1}$ on $H^{1}\times
\mathcal{D}^{1,2}$.

The functional $F$ is strongly indefinite. To avoid this
difficulty, we reduce the study of (\ref{F}) to the study of a
functional in the only variable $u$, as it has been done by the
aforementioned authors.

Now we need some technical results.

\begin{proposition}\label{Propriedade-phi} For every $u \in H^{1}$, there exists an unique
$\phi=\Phi[u]\in \mathcal{D}^{1,2}$ which solves (\ref{system2}).
Furthermore, in the set $\{x|u(x)\neq 0\}$ we have
$-\omega\leq\Phi[u]\leq 0$ if $\omega>0$ and
 $0\leq\Phi[u]\leq-\omega$ if $\omega<0$.
\end{proposition}

\begin{proof} The proof of the uniqueness of $\Phi[u]\in \mathcal{D}^{1,2}(\mathbb{R}^N)$ is very similar to the one proved in dimension three by \cite{Benci-Fortunato}.

Following the same idea of \cite{D'Aprile-Mugnai}, fix $u\in
H^{1}$ and consider $\omega >0$. If we multiply (\ref{system2}) by
$(\omega+\Phi[u])^{-}=\text{min}\{\omega+\Phi[u],0\}$, which is an
admissible test function, we get
\begin{eqnarray*}
-\int\limits_{\{x|\omega+\Phi[u]<0\}}{|\nabla\Phi[u]|^{2}}-\int\limits_{\{x|\omega+\Phi[u]<0\}}
{(\omega+\Phi[u])^{2} u^{2}}= 0
\end{eqnarray*}

\noindent so that $\Phi[u]\geq -\omega$ where $u\neq 0$.
Otherwise, if $\omega<0$ and multiplying (\ref{system2}) by
$(\omega+\Phi[u])^{+}=\text{max}\{\omega+\Phi[u],0\}$ and
repeating the same argument, we obtain $\Phi[u]\leq -\omega$, for
$u\neq 0$.

Finally observe that by Stampacchia's lemma, if $\omega>0$ then
$\phi\leq 0$, and if $\omega<0$, $\phi\geq 0$ (for details see
\cite{Cassani} or \cite{Cunha}).

\end{proof}

In view of Proposition \ref{Propriedade-phi}, we can define
\begin{eqnarray*} \Phi :H^{1}\rightarrow \mathcal{D}^{1,2}
\end{eqnarray*}
\noindent which is of class $C^1$ (see
\cite{D'Aprile-Mugnai-nonexistece}) and maps each $u\in H^{1}$ in
the unique solution of (\ref{system2}), then
\begin{eqnarray}\label{system2-modificado}-\Delta \Phi[u] + u^{2} \Phi[u] = -\omega u^{2}.
\end{eqnarray}

Taking the product of (\ref{system2-modificado}) with $\Phi[u]$
and integrating by parts, we obtain
\begin{eqnarray}\label{system2-modificado2}\int_{\mathbb{R}^{N}}{| \nabla\Phi[u]|^{2}}dx = -\int_{\mathbb{R}^{N}}{\omega
u^{2}\Phi[u]}dx-\int_{\mathbb{R}^{N}}{u^{2}\Phi[u]^{2}}dx.
\end{eqnarray}

Now consider the functional
\begin{eqnarray*}
 J: H^{1}\rightarrow \mathbb{R}, \hspace{0.7cm} J(u) := F(u,\Phi[u])
 \end{eqnarray*}
which is also of class $C^{1}$.

By the definition of $F$ and using (\ref{system2-modificado2}),
$J$ can be written in the following form
\begin{eqnarray} \label{J}
 J(u)=&&\hspace{-0.5cm}\frac{1}{2}\int_{\mathbb{R}^{N}}{\Big(|\nabla u|^{2}+ (m_{0}^{2}-\omega^{2})u^{2}+|\nabla\Phi[u]|^{2}+
\Phi[u]^{2}u^{2}\Big)dx}
\hspace{0.1cm} + \nonumber\\
&& \hspace{-0.5cm}-\frac{\mu}{q}\int_{\mathbb{R}^{N}}{|u|^{q}}dx -
\frac{1}{2^ {*}}\int_{\mathbb{R}^{N}}{|u|^{2^{*}}}dx,
\end{eqnarray}
\noindent while for $J'$ we have, $\forall v\in H^{1}$,
\begin{eqnarray}\label{J'}
\hspace{-1cm}\langle J'(u),v\rangle =\nonumber\\
&&\hspace{-2.5cm}= \int_{\mathbb{R}^{N}}{\Big(\nabla u\cdot \nabla
v+ [m_{0}^{2}-(\omega+\Phi[u])^{2}]uv- \mu |u|^{q-2}uv-|u|^{2^
 {*}-2}uv\Big)}dx\hspace{0.1cm}.
\end{eqnarray}

\begin{proposition} Let $(u,\phi)\in H^{1}\times \mathcal{D}^{1,2}$. Then
the following statements are equi\-valent:
\begin{itemize}
    \item [i)]$(u,\phi)$ is a critical point for $F$;
    \item [ii)]$u$ is a critical point for $J$ and $\phi=\Phi[u]$.
\end{itemize}
\end{proposition}

\begin{proof} See \cite{Benci-Fortunato}.
\end{proof}

Hence, we look for critical points of $J$.


\section{Proof of the Main Result}

In order to overcame the lack of compactness due to the invariance
under group of translations of $J$, we restrict ourselves to
radial functions. More precisely, we look at the functional $J$
on the subspace
\begin{eqnarray*}H_{r}^1 = \{u\in H^1 : u(x)=u(|x|)\}
\end{eqnarray*}
compactly embedded into $L_{r}^p $, $2<p<2^*$, where
$L_{r}^p=\{u\in L^p : u(x)=u(|x|)\} $.

We also point out that any critical point $u\in H_{r}^1$ of
$J|_{H_{r}^1}$ is also a critical point of $J$ by the Principle of
symmetric criticality of Palais (see \cite{Willem}).

Now we show that the functional $J$ verifies the Mountain-Pass
Geometry, more exactly $J$ satisfies the following lemma
\begin{lemma} The functional $J$ satisfies
 \begin{itemize}
 \item [(i)] There exist positive constants $\alpha,\rho$ such that $J(u)\geq\alpha$ for
 $\|u\|=\rho$.
 \item [(ii)] There exists $u_{1}\in H_{r}^1(\mathbb{R}^N)$ with $\|u_{1}\|>\rho$ such that
 $J(u_{1})< 0$.
\end{itemize}
\end{lemma}

\begin{proof} Using the Sobolev embeddings, we have
\begin{eqnarray*}
 J(u)\geq C_{1}\|u\|^2-C_{2}\|u\|^q-C_{3}\|u\|^{2^*},
\end{eqnarray*}
where $C_{1}$, $C_{2}$ and $C_{3}$ are positive constants. Since
$q>2$, there exists $\alpha,\rho >0$ such that
$\inf\limits_{\|u\|=\rho}J(u)>\alpha$, showing $(i)$.

Let $u\in H_{r}^1$, then for $t\geq 0$
\begin{eqnarray}&&J(tu) = \frac{t^2}{2}\int_{\mathbb{R}^N}\Big( |\nabla u|^2+(m_{0}^2-\omega^2)u^2 \Big)dx+ \frac{1}{2}\int_{\mathbb{R}^N}\Big( |\nabla \Phi[tu]|^2+\Phi[tu]^2(tu)^2 \Big)dx +\nonumber\\
&&\hspace{1.5cm}-\frac{\mu}{q}t^q\int_{\mathbb{R}^N}|u|^q dx
-\frac{t^{2^*}}{2^*}\int_{\mathbb{R}^N}|u|^{2^{*}} dx.
\end{eqnarray}

By Proposition (\ref{Propriedade-phi}) we get the estimate
\footnote{From now on we use $\Phi$ ($\Phi_{n}$) instead of
$\Phi[u]$ ($\Phi[u_{n}]$).}
\begin{eqnarray*}
-\int_{\mathbb{R}^{N}}{\omega u^{2}\Phi[u]}dx \leq
\int_{\mathbb{R}^{N}}{\omega ^2 u^2}dx,
\end{eqnarray*}
then using equation (\ref{system2-modificado2}) and the last
inequality in (\ref{J}), we obtain
\begin{eqnarray*}
 J(tu)\leq C_{4} t^2\|u\|^2 +\frac{\omega^2}{2}t^2\|u\|_{2}^2-\frac{\mu}{q}t^q\|u\|_{q}^q-\frac{1}{2^*}t^{2^*}\|u\|_{2^*}^{2^*}.
\end{eqnarray*}

Since $q>2$, there exists $u_{1}\in H_{r}^1$, $u_{1}:=tu$ with $t$
sufficiently large such that $\|u_{1}\|>\rho$ and $J(u_{1})< 0$,
proving (ii).
\end{proof}

Applying a variant of the Ambrosetti-Rabinowitz Mountain Pass
Theorem \cite{Ambrosetti-Rabinowitz} we obtain a ($(PS)_{c}$)
sequence $\{u_{n}\}\subset H_{r}^1$ such that
\begin{eqnarray*}J(u_{n})\rightarrow c \hspace{0.2cm}\text{and}\hspace{0.2cm} J'(u_{n})\rightarrow 0,
\end{eqnarray*}
where
\begin{eqnarray}\label{valorC}c:=\inf_{\gamma\in\Gamma}\max_{0\leq t\leq 1}J(\gamma(t)), \hspace{0.2cm} c\geq\alpha
\end{eqnarray}
\noindent and
\begin{eqnarray}\Gamma=\{\gamma\in\mathcal{C}([0,1],H_{r}^{1}(\mathbb{R}^{N}))|\gamma(0)=0, \gamma(1)=u_{1} \}.
\end{eqnarray}

An important tool in our analysis will be the next lemma:

\begin{lemma}
The $(PS)_{c}$ sequence $\{u_{n}\}$ is bounded.
\end{lemma}

\begin{proof}
By hypothesis, let $\{u_{n}\}\subset H_{r}^{1}$ be such that
$-\langle J'(u),v\rangle \leq o(1)\|u_{n} \|$ and $|J(u_{n})|\leq
M$, for some positive constant M. Then from (\ref{J}) and
(\ref{J'}),
\begin{eqnarray}\label{ps-bounded-inequality}
&&\hspace{-0.5cm} qM + o(1)\|u_{n} \| \geq  q J(u_{n})-\langle
J'(u_{n}),u_{n}\rangle = \nonumber \\
&&= \Big(\frac{q}{2}-1\Big)\int_{\mathbb{R}^{N}}{\Big( |\nabla
u_{n}|^{2}+ [m_{0}^{2}-\omega^{2}]u_{n}^{2}\Big)dx }
+\Big(2-\frac{q}{2}\Big) \int_{\mathbb{R}^{N}}{\omega
\Phi_{n}u_{n}^{2}}dx+\int_{\mathbb{R}^{N}}{\Phi_{n}^{2}u_{n}^{2}}dx
+\nonumber\\
&&\hspace{0.5cm}
+\Big(1-\frac{q}{2^{*}}\Big)\int_{\mathbb{R}^{N}}{|u_{n}|^{2^
{*}}}dx \nonumber\\
&& \geq \Big(\frac{q-2}{2}\Big)\int_{\mathbb{R}^{N}}{\Big( |\nabla
u_{n}|^{2}+ [m_{0}^{2}-\omega^{2}]u_{n}^{2}\Big) }dx
-\omega\Big(\frac{q-4}{2}\Big) \int_{\mathbb{R}^{N}}{
\Phi_{n}u_{n}^{2}}dx.
\end{eqnarray}

There are two cases to be considered: either $2<q<4$ or  $4\leq q
< 2^{*}$.

If $4\leq q < 2^{*}$, then by Proposition \ref{Propriedade-phi}
and inequality (\ref{ps-bounded-inequality})
\begin{eqnarray*}
qM + o(1)\|u_{n} \| &\geq & C\|u_{n} \|^{2}+
\omega\Big(\frac{q-4}{2}\Big)\int_{\mathbb{R}^{N}}{(-\Phi_{n})u_{n}^{2} }dx\\
&\geq& C\|u_{n} \|^{2}
\end{eqnarray*}

\noindent and we deduce that $\{u_{n} \}$ is bounded in
$H_{r}^{1}$.

But if $2<q<4$ and using again (\ref{ps-bounded-inequality}) and
Proposition \ref{Propriedade-phi} we get
\begin{eqnarray*}
qM + o(1)\|u_{n} \| &\geq &
\Big(\frac{q-2}{2}\Big)\int_{\mathbb{R}^{N}}{ |\nabla
u_{n}|^{2}}dx+ \Big(\frac{(q-2)m_{0}^{2}-2\omega^{2}}{2}
 \Big)\int_{\mathbb{R}^{N}}{|u_{n}^{2}|}dx\\
&\geq& C\|u_{n} \|^{2},
\end{eqnarray*}

\noindent where $(q-2)m_{0}^{2}-2\omega^{2}>0$ by hypothesis,
which implies that $\{u_{n} \}$ is again bounded in $H_{r}^{1}$.
\end{proof}

In view of the previous lemma we have that $\{\Phi_{n}\}$ is
bounded in $\mathcal{D}_{r}^{1,2}$ because
\begin{eqnarray*}\|\Phi_{n}\|_{\mathcal{D}_{r}^{1,2}}^2 &\leq & \int_{\mathbb{R}^N} |\nabla\Phi_{n}|^2 dx + \int_{\mathbb{R}^N} |\Phi_{n}^2 u_{n}^2| dx \\
 &=& -\omega \int_{\mathbb{R}^N} |\Phi_{n} u_{n}^2| dx \leq C \omega \|\Phi_{n}\|_{\mathcal{D}_{r}^{1,2}}
 \|u_{n}\|_{2\cdot 2^{*}/(2^{*} -1)}^2.
\end{eqnarray*}

So, passing to a subsequence if necessary, we may assume

\begin{itemize}
    \item []\hspace{3cm} $u_{n}\rightharpoonup u$,\hspace{0.3cm} weakly in $H_{r}^{1}$, \hspace{0.1cm} $n\rightarrow\infty$,
    \item [] \hspace{3cm} $\Phi_{n}\rightharpoonup \phi$,\hspace{0.3cm} weakly in
    $\mathcal{D}_{r}^{1,2}$, \hspace{0.1cm} $n\rightarrow\infty$.
\end{itemize}

\begin{lemma}\label{existence-of-Phi} $\phi=\Phi[u]$ and $\Phi_{n}\rightarrow \Phi$
strongly in $\mathcal{D}_{r}^{1,2}$.
\end{lemma}

\begin{proof}
 The proof is essentially as in Lemma 3.2 of \cite{Cassani}, which can be easily extended in dimension $N$.
\end{proof}

Moreover, since the Sobolev embeddings $H_{r}^{1} \hookrightarrow
L_{r}^s$, $2 < s < 2^*$, are compact we conclude that
\begin{itemize}
    \item []\hspace{3cm} $u_{n}\rightarrow u$,\hspace{0.3cm} strongly in $L_{r}^{s}$,\hspace{0.1cm}
 for $2<s<2^{*}$, \hspace{0.1cm} $n\rightarrow\infty$.
\end{itemize}

Now we show that the pair ($u, \Phi$) satisfies the
$(\mathcal{KGM})$ system in the weak sense. Indeed, since
$J'(u_{n})\rightarrow 0$ we have $\forall v\in H_{r}^1$,
\begin{eqnarray}\label{eq}&&\hspace{-1cm}\int_{\mathbb{R}^N} \Big( \nabla u_{n}\nabla v+(m_{0}^2-\omega^2)u_{n}v \Big)dx =
\int_{\mathbb{R}^N}u_{n}\Phi_{n}^2 v dx + 2\omega\int_{\mathbb{R}^N}\Phi_{n}u_{n} v dx +\nonumber\\
&&\hspace{3.8cm}+ \mu \int_{\mathbb{R}^N}|u_{n}|^{q-2}u_{n} v dx+
\int_{\mathbb{R}^N}|u_{n}|^{2^*-2}u_{n} v dx + o(1)
\end{eqnarray}

We will prove that
\begin{eqnarray}\label{eq1}\int_{\mathbb{R}^N}u_{n}\Phi_{n}^2 v dx + 2\omega\int_{\mathbb{R}^N}\Phi_{n}u_{n} v dx
\stackrel{n\rightarrow\infty}{\longrightarrow}
\int_{\mathbb{R}^N}u\Phi^2 v dx + 2\omega\int_{\mathbb{R}^N}\Phi u
v dx ,
\end{eqnarray}
\begin{eqnarray}\label{eq2}\int_{\mathbb{R}^N}|u_{n}|^{q-2}u_{n} v dx
\stackrel{n\rightarrow\infty}{\longrightarrow}
\int_{\mathbb{R}^N}|u|^{q-2}uv dx
\end{eqnarray}
and
\begin{eqnarray}\label{eq3}\int_{\mathbb{R}^N}|u_{n}|^{2^*-2}u_{n} v dx
\stackrel{n\rightarrow\infty}{\longrightarrow}
\int_{\mathbb{R}^N}|u|^{2^*-2}u v dx
\end{eqnarray}

\vspace{0.3cm}

\noindent \textit{Verification of (\ref{eq1})}.

Using the generalized H\"{o}lder inequality, note that
\begin{eqnarray*}&&\hspace{-3.9cm}\int_{\mathbb{R}^N}|\Phi u -\Phi_{n}u_{n}| |v| dx \leq \|\Phi-\Phi_{n}\|_{2^*}  \|u\|_{2\cdot 2^*/(2^*-2)}   \|v\|_{2} +\\
&&\hspace{0.3cm}+\|\Phi_{n}\|_{2^*}  \|u-u_{n}\|_{2\cdot
2^*/(2^*-2)}   \|v\|_{2}
\end{eqnarray*}
and
\begin{eqnarray*}&&\hspace{-1.2cm}\int_{\mathbb{R}^N}|u\Phi^2 -u_{n}\Phi_{n}^2| |v| dx \leq \|\Phi-\Phi_{n}\|_{2^*}  \|\Phi+\Phi_{n}\|_{2^*}\|u_{n}\|_{2^*}  \|v\|_{2} +\\
&&\hspace{3cm}+\|u-u_{n}\|_{2\cdot 2^*/(2^*-2)}
\|\Phi^2\|_{2\cdot 2^*/(2^*-2)}\|v\|_{2\cdot 2^*/(2^*-2)}.
\end{eqnarray*}

Then, by Lemma \ref{existence-of-Phi} we get (\ref{eq1}).

\vspace{0.3cm}

\noindent \textit{Verification of (\ref{eq2})-(\ref{eq3})}.

The convergence in (\ref{eq2}) follows from the compactness of the
embedding $H_{r}^{1} \hookrightarrow L_{r}^q $ and the convergence
in (\ref{eq3}) holds since $\{u_{n}\}$ is bounded in
$L^{2^*}_{r}$.

\vspace{0.3cm}

Hence by (\ref{eq1}), (\ref{eq2}) and (\ref{eq3}) together with
(\ref{eq}), we conclude that $(u,\Phi)$ is a weak solution for
$(\mathcal{KGM})$ system.

Due to the lack of compactness, we must prove that actually $u$
does not vanish.

\begin{lemma}\label{c} The number $c$ given in (\ref{valorC}) satisfies
\end{lemma}
\begin{equation} c<\frac{1}{N}S^{N/2},
\end{equation}
\noindent \textit{where S is the best Sobolev constant, namely}
\begin{eqnarray*} S:=\inf\limits_{u\in \mathcal{D}^{1,2}(\mathbb{R}^{N})\atop_{u\neq
0}}\frac{\int|\nabla u|^{2}dx}{\Big( \int{|u|^{2^{*}}}
dx\Big)^{2/2^{*}}}.
\end{eqnarray*}

For a moment, suppose Lemma \ref{c} holds true, we will prove that
$u\neq 0$. Consider $u\equiv 0$. Since $J'(u_{n})\rightarrow 0$
and $u_{n}\rightarrow 0$ in $L^q_{r}$ as $n\rightarrow\infty$, we
may assume
\begin{eqnarray*}\int_{\mathbb{R}^N}\Big(|\nabla u_{n}|^{2}+
(m_{0}^{2}-\omega^{2})u_{n}^{2}\Big)dx
\stackrel{n\rightarrow\infty}{\longrightarrow} \ell
\end{eqnarray*}
and
\begin{eqnarray*}\int_{\mathbb{R}^N}|u_{n}|^{2^*}dx \stackrel{n\rightarrow\infty}{\longrightarrow} \ell, \hspace{0.3cm} \ell \geq 0.
\end{eqnarray*}

Consequently,
\begin{eqnarray*}J(u_{n})\stackrel{n\rightarrow\infty}{\longrightarrow} \Big(\dfrac{1}{2}-\dfrac{1}{2^*}\Big)\ell
\end{eqnarray*}
where now $\ell>0$, since $c>0$.

By the definition of $S$,
\begin{eqnarray*}S\leq \dfrac{\displaystyle\int_{\mathbb{R}^N} \Big( |\nabla u_{n}|^2+(m_{0}^2-\omega^2)u_{n}^2 \Big)dx}{\Big(\int{|u|^{2^{*}}}dx
\Big)^{2/2^{*}}}\stackrel{n\rightarrow\infty}{\longrightarrow}
\ell^{2/N},
\end{eqnarray*}
from what we conclude that
\begin{eqnarray*}c=\Big(\dfrac{1}{2}-\dfrac{1}{2^*}\Big)\ell\geq \frac{1}{N}S^{N/2}
\end{eqnarray*}
contradicting Lemma \ref{c}.

\vspace{0.3cm}

\noindent \textit{Proof of Lemma \ref{c}}. This proof uses a
technique by Br\'{e}zis and Nirenberg \cite{Brezis-Nirenberg} and
some of its variants. However we follow more closely Miyagaki
\cite{Olimpio}.

It suffices to show that
\begin{eqnarray}\label{suffices}\sup_{t\geq 0} J(tv_{0})< \frac{1}{N}
S^{\frac{N}{2}}
\end{eqnarray}
\noindent for some $v_{0}\in H^{1}_{r}, v_{0}\neq 0 $.

Indeed, observing that $J(tv_{0})\rightarrow -\infty$ as
$t\rightarrow\infty$ and letting $\gamma\in\Gamma$ we have
\begin{eqnarray}\label{sup}J(\gamma(t))\leq \sup_{t\geq 0} J(tv_{0}),
\hspace{0.3cm} 0\leq t\leq 1
\end{eqnarray}
so that
\begin{eqnarray*}c\leq\sup_{t\geq 0} J(tv_{0})<\frac{1}{N}
S^{\frac{N}{2}}.
\end{eqnarray*}

In order to prove (\ref{sup}) consider $R>0$ fixed and a cut-off
function $\varphi\in C_{0}^{\infty}$ such that
\begin{eqnarray*}\varphi|B_{R}=1, \hspace{0.3cm} 0\leq\varphi\leq
1 \hspace{0.1cm}\text{in}\hspace{0.1cm} B_{2R}\hspace{0.3cm}
\text{and}\hspace{0.3cm} \text{supp}\hspace{0.06cm} \varphi\subset
B_{2R}.
\end{eqnarray*}

Let $\varepsilon >0$ and define
$w_{\varepsilon}:=u_{\varepsilon}\varphi$ where
$u_{\varepsilon}\in \mathcal{D}^{1,2}$ is the well known Talenti's
function (see \cite{Talenti})
\begin{eqnarray*}u_{\varepsilon}(x)=\frac{[N(N-2)\varepsilon]^{\frac{N-2}{4}}}{\Big( \varepsilon+|x|^{2} \Big)^{\frac{N-2}{2}}},
\hspace{0.3cm} x\in\mathbb{R}^{N}, \varepsilon > 0
\end{eqnarray*}

\noindent and also consider $v_{\varepsilon}\in C_{0}^{\infty}$
given by
\begin{eqnarray}\label{v-epsilon definicao}v_{\varepsilon}:=\frac{w_{\varepsilon}}{\|w_{\varepsilon}\|_{L^{2^{*}}(B_{2R})}}.
\end{eqnarray}

From the estimates given in \cite{Brezis-Nirenberg} we have, as
$\varepsilon\rightarrow 0$,
\begin{eqnarray}\label{X-epsilon} X_{\varepsilon}:=\|\nabla
v_{\varepsilon}\|_{2}^{2}\leq S+O(\varepsilon^{\delta}),
\hspace{0.3cm} \text{where}\hspace{0.1cm} \delta=\frac{N-2}{2}.
\end{eqnarray}

Since $\lim \limits_{t\rightarrow\infty}
J(tv_{\varepsilon})=-\infty \hspace{0.1cm} \forall\varepsilon$,
there exists $t_{\varepsilon}\geq 0$ such that $\sup\limits_{t\geq
0}J(tv_{\varepsilon})=J(t_{\varepsilon}v_{\varepsilon })$ and we
may assume without loss of generality that $t_\varepsilon\geq
C_{0}>0$.

\vspace{0.3cm}

\noindent\textbf{Claim 1}. The following estimate holds
\begin{eqnarray}t_{\varepsilon}\leq \Big( \int_{B_{2R}}|\nabla v_{\varepsilon}|^{2}dx+\int_{B_{2R}} m_{0}^{2}v_{\varepsilon}^{2}dx
\Big)^{1/(2^{*}-2)}:=r_{\varepsilon}.
\end{eqnarray}

\noindent\textit{Proof of Claim 1}: Letting
$\gamma(t):=J(tv_{\varepsilon})$ we have, for $t >
r_{\varepsilon}$,
\begin{eqnarray*}
\gamma'(t)&=& J'(tv_{\varepsilon})(v_{\varepsilon})\\
&=& t r_{\varepsilon}^{2^{*}-2}-t^{2^{*}-1}
-t\int_{B_{2R}}(\omega+\phi[tv_{\varepsilon}])^{2}v_{\varepsilon}^{2}dx-\mu
t^{q-1}\int_{B_{2R}}|v_{\varepsilon}|^{q}dx\\
&<& 0.
\end{eqnarray*}

\hspace{12.5cm}$\Box$

Now, the function of $t$
\begin{eqnarray*}\frac{t^{2}}{2}r_{\varepsilon}^{2^{*}-2}-\frac{t^{2^{*}}}{2^{*}}
\end{eqnarray*}
\noindent is increasing on $[0,r_{\varepsilon})$, hence using
(\ref{X-epsilon}) we conclude that
\begin{eqnarray*}&&\hspace{-0.5cm}J(t_{\varepsilon}v_{\varepsilon}) \leq
\frac{1}{N}\Big(S+O(\varepsilon^{\delta})+\int_{B_{2R}}m_{0}^{2}v_{\varepsilon}^{2}dx
\Big)^{N/2} -
\frac{t_{\varepsilon}^{2}}{2}\int_{B_{2R}}\omega^{2}v_{\varepsilon}^{2}dx+\\
&&\hspace{0.9cm}+
Ct_{\varepsilon}^{4}\|v_{\varepsilon}\|^4_{2\cdot 2^{*}/(2^{*}-1)}
-
\frac{\mu}{q}t_{\varepsilon}^{q}\int_{B_{2R}}|v_{\varepsilon}|^{q}dx.
\end{eqnarray*}

Recalling that
\begin{eqnarray*}(a+b)^{\alpha}\leq
a^{\alpha}+\alpha(a+b)^{\alpha-1}b
\end{eqnarray*}
which is valid for $a,b\geq 0$, $\alpha\geq 1$ we obtain
\begin{eqnarray*}&&\hspace{-0.5cm} J(t_{\varepsilon}v_{\varepsilon}) \leq
\frac{1}{N}S^{N/2}+O(\varepsilon^{\delta})
+K_{1}\int_{B_{2R}}{m_{0}^{2}v_{\varepsilon}^{2}}dx+
\\
&&\hspace{0.9cm}
-K_{2}\int_{B_{2R}}{\omega^{2}v_{\varepsilon}^{2}}dx
 -\mu
K_{3}\int_{B_{2R}}{|v_{\varepsilon}|^{q}}dx +
K_{4}\|v_{\varepsilon}\|^{4}_{2\cdot 2^{*}/(2^{*}-1)},
\end{eqnarray*}
where $K_{i}(\varepsilon)\geq K_{0}>0$.

We contend that

\vspace{0.3cm}

\noindent\textbf{Claim 2}.
\begin{eqnarray}\label{star} \lim_{\varepsilon\rightarrow 0}\frac{1}{\varepsilon^{\delta}}\Big( \int_{B_{2R}}(v_{\varepsilon}^{2}-\mu
v_{\varepsilon}^{q})dx +\|v_{\varepsilon}\|_{2\cdot
2^{*}/(2^{*}-1)}^{4}  \Big)=-\infty.
\end{eqnarray}

Assuming (\ref{star}) for a while we have
\begin{eqnarray*}J(t_{\varepsilon}v_{\varepsilon})<
\frac{1}{N}S^{N/2}, \hspace{0.3cm} \varepsilon\hspace{0.1cm}
\text{small}
\end{eqnarray*}
showing (\ref{suffices}) and thus Lemma \ref{c}.

\vspace{0.3cm}

\noindent\textit{Proof of Claim 2}:

As in \cite{Brezis-Nirenberg}, we obtain
\begin{eqnarray}\int_{B_{2R}}|w_{\varepsilon}|^{2^{*}}dx=(N(N-2))^{N/2}\int_{\mathbb{R}^{N}}\frac{1}{(1+|x|^{2})^{N}}dx+O(\varepsilon^{N/2})
\end{eqnarray}
so, in view of (\ref{v-epsilon definicao}), it suffices evaluate
(\ref{star}) with $w_{\varepsilon}$ instead of $v_{\varepsilon}$.
In order to prove (\ref{star}) we must show
\begin{eqnarray}\label{star2}\lim\limits_{\varepsilon\rightarrow
0}\frac{1}{\varepsilon^{\delta}}\Big[
\int_{B_{R}}(w_{\varepsilon}^{2}-\mu w_{\varepsilon}^{q})dx
+\Big(\int_{B_{R}}|w_{\varepsilon}|^{\frac{4N}{N+2}}dx
\Big)^{\frac{N+2}{N}} \Big]=-\infty
\end{eqnarray}
and also that
\begin{eqnarray}\label{star22}\dfrac{1}{\varep^\delta}\Big[\int_{B_{2R}\setminus B_{R}}(v_{\varepsilon}^{2}-\mu
v_{\varepsilon}^{q})dx +\Big(\int_{B_{2R}\setminus
B_{R}}|v_{\varepsilon}|^{\frac{4N}{N+2}}dx
\Big)^{\frac{N+2}{N}}\Big]
\end{eqnarray}
is bounded.

\vspace{0.3cm}

\noindent \textit{Verification of (\ref{star2})}. Let
\begin{eqnarray*}
I_{\varepsilon}:=\frac{1}{\varepsilon^{\delta}}\Big[
\int_{B_{R}}(w_{\varepsilon}^{2}-\mu w_{\varepsilon}^{q})dx
+\Big(\int_{B_{R}}|w_{\varepsilon}|^{\frac{4N}{N+2}}dx
\Big)^{\frac{N+2}{N}}\Big]
\end{eqnarray*}

On $B_{R}$, by changing variables we have (see \cite{Cunha})
\begin{eqnarray}\label{aqui}
I_{\varepsilon}&\leq & \varepsilon^{1-\delta}\Big[ C_{1}
\int_{0}^{\frac{R}{\sqrt{\varep}}}\frac{r^{N-1}}{(1+r^{2})^{N-2}}
dr - \mu C_{2}\varep^{-\frac{(N-2)}{4}q+\frac{N}{2}-1}
\int_{0}^{\frac{R}{\sqrt{\varep}}}\frac{r^{N-1}}{(1+r^{2})^{(N-2)q/2}
} dr \nonumber\\
&+& C_{3}\varep^{ \frac{4-N}{2}}\Big(
\int_{0}^{\frac{R}{\sqrt{\varep}}}\frac{r^{N-1}}{(1+r^{2})^{\frac{2N(N-2)}{N+2}}
} dr \Big)^{\frac{N+2}{N}} \Big]
\end{eqnarray}
where $C_{i}$ depends only on $N$.

Now we distinguish the cases $N\geq 6$, $N=5$, $N=4$ and $N=3$ as
follows:

\vspace{0.3cm}

\noindent \textit{Case 1}. $N\geq 6$

It is not difficult to see that for $N\geq 6 $ and $q>2$ all
integrals in (\ref{aqui}) are convergent as $\varep\rightarrow 0$.
Besides we also have
$-\frac{(N-2)}{4}q+\frac{N}{2}-1>\frac{4-N}{2}$ for $2<q<2^*$,
then $I_{\varep}\rightarrow -\infty$ as $\varep\rightarrow 0$,
proving (\ref{star2}).

\vspace{0.3cm}

\noindent \textit{Case 2}. $N=5$

As in \textit{Case 1} all integrals in (\ref{aqui}) are convergent
as $\varep\rightarrow 0$  for $N=5 $ and $2<q<2^*$.

There are two cases to be considered: either $2<q<\frac{8}{3}$ or
$\frac{8}{3}\leq q<2^*$. For $2<q<\frac{8}{3}$ we immediately see
that $-\frac{(N-2)}{4}q+\frac{N}{2}-1>\frac{4-N}{2}$ and for
$\frac{8}{3}\leq q<2^*$ we choose $\mu=e^{1/\varep}$. So in both
cases we get $I_{\varep}\rightarrow -\infty$ as $\varep\rightarrow
0$.

\vspace{0.3cm}

\noindent \textit{Case 2}. $N=4 $

Using the fact that $q<2^{*}=4$ and by computing
\begin{eqnarray*}\int_{0}^{\frac{R}{\sqrt{\varep}}}\frac{r^{3}}{(1+r^{2})^{2}}
dr=\frac{1}{2}\Big(\log(1+\frac{R^{2}}{\varep})+\frac{\varep}{\varep+R^{2}}-1\Big)
\end{eqnarray*}
and
\begin{eqnarray*}
\int_{0}^{\frac{R}{\sqrt{\varep}}}\frac{r^{3}}{(1+r^{2})^{4}}
dr=\frac{1}{12}-\frac{\varep^{2}(\varep+3R^{2})}{12(\varep+R^{2})^{3}}
\end{eqnarray*}
we get
\begin{eqnarray*}I_{\varep}\leq &&
\hspace{-0.5cm}\frac{C_{1}}{2}\Big(\log(1+\frac{R^{2}}{\varep})+\frac{\varep}{\varep+R^{2}}-1\Big)-\mu
C_{2}\varep^{\frac{2-q}{2}}\Big(\frac{1}{12}-\frac{\varep^{2}(\varep+3R^{2})}{12(\varep+R^{2})^{3}}\Big)+\\
&&\hspace{-0.5cm}+C_{3}\Big(\int_{0}^{\frac{R}{\sqrt{\varep}}}\frac{r^3}{(1+r^2)^{8/3}}
\Big)^{3/2}
\end{eqnarray*}

But since
\begin{eqnarray*}\lim_{\varep\rightarrow
0}\frac{\varep^{\frac{2-q}{2}}}{\log(1+\frac{R^{2}}{\varep})}=+\infty
\end{eqnarray*}
we conclude that $I_{\varep}\rightarrow -\infty$ as
$\varep\rightarrow 0$.

\vspace{0.3cm}

\noindent \textit{Case 3}. $N=3 $

By simple computations, one gets
\begin{eqnarray*}\int_{0}^{\frac{R}{\sqrt{\varep}}}\frac{r^{2}}{1+r^{2}}
dr=\frac{R}{\sqrt{\varep}}-\arctan(\frac{R}{\sqrt{\varep}})
\end{eqnarray*}
then, arguing as in the proof of the case $N=4$,
\begin{eqnarray*}I_{\varep}&\leq &
C_{1}R-C_{1}\varep^{1/2}\arctan(\frac{R}{\sqrt{\varep}}) -\mu
C_{2}\varep^{\frac{4-q}{4}}\int_{0}^{\frac{R}{\sqrt{\varep}}}\frac{r^{2}}{(1+r^{2})^{q/2}}
dr +\\
&&+C_{3}\varep\Big(\int_{0}^{\frac{R}{\sqrt{\varep}}}\frac{r^{2}}{(1+r^{2})^{6/5}}
dr  \Big)^{5/3}\\
&\leq & C_{1}R-\mu
C_{2}\varep^{\frac{4-q}{4}}\int_{0}^{\frac{R}{\sqrt{\varep}}}\frac{r^{2}}{(1+r^{2})^{q/2}}
dr+C_{3}R^{5/3}\varep^{1/6}
\end{eqnarray*}

We have to distinguish two cases: either $2<q\leq 4$ or $4<q<2^*$.

The case $4<q<2^*$ was proved by Cassani \cite{Cassani}. However,
we can also show (\ref{star2}) using the last inequality, since
the integral
$\textstyle\int_{0}^{\infty}\frac{r^{2}}{(1+r^{2})^{q/2}} dr$
converges.

If $2<q\leq 4$ and noting that
$\textstyle\int_{0}^{\infty}\frac{r^{2}}{(1+r^{2})^{q/2}} dr\geq
\frac{\pi}{4}$ we conclude
\begin{eqnarray*}I_{\varep}\leq C_{4}-\frac{\pi}{4}\mu
C_{2}\varep^{\frac{4-q}{4}}
\end{eqnarray*}
Finally, choosing $\mu=\varep^{-\frac{1}{2}}$, we infer that
$I_{\varep}\rightarrow -\infty$ as $\varep\rightarrow 0$.

Hence this proves (\ref{star2}).

\vspace{0.3cm}

\noindent \textit{Verification of (\ref{star22})}. We have
\begin{eqnarray*}&&\hspace{-1cm}\frac{1}{\varep^\delta}\Big[\int_{B_{2R}\setminus B_{R}}(v_{\varepsilon}^{2}dx-\mu
v_{\varepsilon}^{q})dx +\Big(\int_{B_{2R}\setminus
B_{R}}|v_{\varepsilon}|^{2\cdot 2^{*}/(2^{*}-1)}dx \Big)^{2\cdot
(2^{*}-1)/2^{*}}\Big] \leq \\
&\leq & \frac{C_{1}}{\varep^{\delta}}\int_{B_{2R}\setminus
B_{R}}\varphi^{2}u_{\varep}^{2}dx+\frac{C_{3}}{\varep^{\delta}}\Big(\int_{B_{2R}\setminus
B_{R}}\varphi^{2\cdot 2^{*}/(2^{*}-1)}|u_{\varepsilon}|^{2\cdot
2^{*}/(2^{*}-1)} dx\Big)^{2\cdot
(2^{*}-1)/2^{*}}\\
&\leq & C_{1}\varep \|\varphi\|^{2}_{H^{1}(B_{2R}\setminus B_{R})}
+ C_{2}\varep^{2+\delta}\|\varphi^{2^{*}/(2^{*}-1)} \|^{2\cdot
(2^{*}-1)/2^{*}}_{H^{1}(B_{2R}\setminus B_{R})}
\end{eqnarray*}
where we choose R large such that $u_{\varep}^{2}\leq
\varep^{1+\delta}$, $\forall |x|\geq\delta$. Then we conclude that
equation (\ref{star22}) is bounded.

Consequently, the proof of Claim 2 is complete.
\hspace{4.8cm}$\Box$




\end{document}